\documentclass{siamart190516}

\usepackage{graphicx}
\usepackage{fullpage}
\usepackage{hyperref}
\usepackage{graphicx}
\usepackage{amsmath}
\usepackage{xcolor}
\usepackage{amssymb}
\usepackage{amsfonts}
\usepackage{comment}
\usepackage{marginnote}
\usepackage{bm}
\usepackage{float}

\allowdisplaybreaks

\def\omg{{\Omega}}
\def\omgi{{\omg_I}}
\def\ouo{{\omg\cup\omgi}}

\def\gam{\gamma}

\newcommand{\mcK}{\mathcal{K}}

\newcommand{\mcL}{\mathcal{L}}

\newcommand{\mcA}{\mathcal{A}}

\newcommand{\mcX}{\mathcal{X}}

\newcommand{\mbR}{\mathbb{R}}

\newcommand{\mbRn}{{\mathbb{R}^n}}

\def \xb{\mathbf{x}}

\def \yb{\mathbf{y}}

\newcommand{\abs}[1]{\left|#1\right|}
\renewcommand{\vec}[1]{\mathbf{#1}}
\newcommand{\norm}[1]{\left|\!\left|#1\right|\!\right|}
\newcommand{\opnorm}[2][\gamma]{\left|\!\left|\!\left|#2\right|\!\right|\!\right|_{#1}}
\renewcommand{\tilde}[1]{\widetilde{#1}}
\newcommand{\V}{\tilde{H}_{\delta}^{s}(\Omega)}

\newsiamremark{remark}{Remark}

\begin{document}

\title{A fractional model for anomalous diffusion with increased variability. Analysis, algorithms and applications to interface problems}

\author{
  Marta D'Elia\thanks{Computational Science and Analysis, Sandia National Laboratories, CA, USA, (\email{mdelia@sandia.gov})}
  \and
  Christian Glusa\thanks{Center for Computing Research, Sandia National Laboratories, NM, USA}
}

\maketitle

\begin{abstract}
  Fractional equations have become the model of choice in several applications where heterogeneities at the microstructure result in anomalous diffusive behavior at the macroscale. In this work we introduce a new fractional operator characterized by a doubly-variable fractional order and possibly truncated interactions. Under certain conditions on the model parameters and on the regularity of the fractional order we show that the corresponding Poisson problem is well-posed. We also introduce a finite element discretization and describe an efficient implementation of the finite-element matrix assembly in the case of piecewise constant fractional order. Through several numerical tests, we illustrate the improved descriptive power of this new operator across media interfaces. Furthermore, we present one-dimensional and two-dimensional $h$-convergence results that show that the variable-order model has the same convergence behavior as the constant-order model.
\end{abstract}



\section{Introduction}
Nonlocal models are becoming a popular alternative to partial differential equations (PDEs) when the latter fail to capture effects such as multiscale and anomalous behavior. In fact, several scientific and engineering applications exhibit hierarchical features that cannot be described by classical models. As an example we mention applications in continuum mechanics \cite{ha2011characteristics,littlewood2010simulation,ouchi2015fully},
phase transitions \cite{Bates1999,Burkovska2020CH,Chen_nonlocalmodels},
corrosion \cite{Rokkam2019},
turbulence \cite{Bakunin2008,Pang2020,Schekochihin2008},
and geoscience \cite{Benson2000,Meerschaert2006,Schumer2001,Schumer2003,WeissBloemenEtAl2020_FractionalOperatorsAppliedToGeophysicalElectromagnetics}.
In this work we consider nonlocal operators of fractional type, i.e. integral operators characterized by singular and non-integrable kernels, that correspond to differential operators of fractional orders, as opposed to the integer order in the PDE case. 

Note that fractional differential operators are almost as old as their integer counterpart \cite{Gorenflo1997}; however, their usability has increased during the last decades thanks to progress in computational capabilities and to a better understanding of their descriptive power. The simplest form of a fractional operator is the fractional Laplacian; in its integral form, its action on a scalar function $u$ is defined as \cite{Gorenflo1997}
\begin{equation}\label{eq:fractional-laplacian}
(-\Delta)^s u(\xb) = C_{n,s} \; {\rm p.v.}\int_\mbRn \frac{u(\xb)-u(\yb)}{|\xb-\yb|^{n+2s}} d\yb,
\end{equation}
where $s\in(0,1)$ is the {\it fractional order}, $n$ the spatial dimension, $C_{n,s}$ a constant and where p.v. indicates the principal value. It follows from \eqref{eq:fractional-laplacian} that in fractional modeling the state of a system at a point depends on the value of the state at any other point in the space; in other words, fractional models are nonlocal. Specifically, fractional operators are special instances of more general nonlocal operators \cite{DElia2020,DElia2013,Du2012,Pang2020} of the following form \cite{DElia2017}:
\begin{equation}\label{eq:nonlocal-laplacian}
-\mcL u(\xb) = \int_{B_\delta(\xb)} (u(\xb)\gam(\xb,\yb)-u(y)\gam(\yb,\xb)) d\yb.
\end{equation}
Here, interactions are limited to a Euclidean ball $B_\delta(\xb)$ of radius $\delta$, often referred to as horizon or interaction radius. The kernel $\gam(\xb,\yb)$ is a modeling choice and determines regularity properties of the solution.
Note that for $\delta=\infty$ and for the fractional-type kernel $\gam(\xb,\yb)=|\xb-\yb|^{-n-2s}$ the nonlocal operator in \eqref{eq:nonlocal-laplacian} is equivalent to the fractional Laplacian in \eqref{eq:fractional-laplacian}.
Also, it has been shown in \cite{DElia2013} that for that choice of $\delta$ and $\gam$ solutions corresponding to the nonlocal operator \eqref{eq:nonlocal-laplacian} converge to the ones corresponding to the fractional operator \eqref{eq:fractional-laplacian} as $\delta\to\infty$ (see \cite{DElia2020} for more convergence results and for a detailed classification of these operators and relationships between them).

In recent years, with the purpose of increasing the descriptive power of fractional operators, new models characterized by a variable fractional order have been introduced for both space- and time-fractional differential operators \cite{Razminia2012,AntilRautenberg2019_SobolevSpacesWithNon,Zheng2020} and several discretization methods have been designed \cite{Chen2015,Zeng2015,Zheng2020optimal-order,Zhuang2009,SchneiderReichmannEtAl2010_WaveletSolutionVariableOrderPseudodifferentialEquations}. However, the analysis of variable-order models is still in its infancy, with \cite{Felsinger2015} and \cite{Schilling2015} being perhaps the only relevant works that address theoretical questions such as well-posedness for space-fractional differential operators with variable order $s=s(\xb)$. The improved descriptive power of variable-order operators has been demonstrated in some works on parameter estimation \cite{Pang2020,Pang2019fPINNs,Pang2017discovery,WANG}; here, via machine-learning algorithms or more standard optimization techniques, the authors estimate the variable fractional power for operators of the form
\begin{equation}\label{eq:variable-truncated-Laplacian}
-\mcL u(\xb) = C_{n,s} \; {\rm p.v.}\int_{B_\delta(\xb)} \frac{u(\xb)-u(\yb)}{|\xb-\yb|^{n+2s(\xb)}} d\yb,
\end{equation}
where $\delta$ is either infinite or finite. 

Models such as the one in \eqref{eq:variable-truncated-Laplacian} are convenient to describe anomalous diffusion in case of heterogeneous materials or media, where different regions may be characterized by different diffusion rates, i.e. different values of $s$, see Figure \ref{fig:interface-domain-config} for two-dimensional illustrations where $s(\xb)$ is a piece-wise constant function defined over the domain. In other words, these models can describe physical interfaces by simply tuning the function $s(\xb)$. Modeling nonlocal interfaces is a nontrivial task due to the unknown nature of the nonlocal interactions across the interfaces. Only a few works in the literature have addressed this problem; among them, we mention \cite{Alali2015} for nonlocal models in mechanics and \cite{Capodaglio2020} for nonlocal diffusion models with integrable kernels. 
\begin{figure}[t]
\centerline{
\includegraphics[height=2.in]{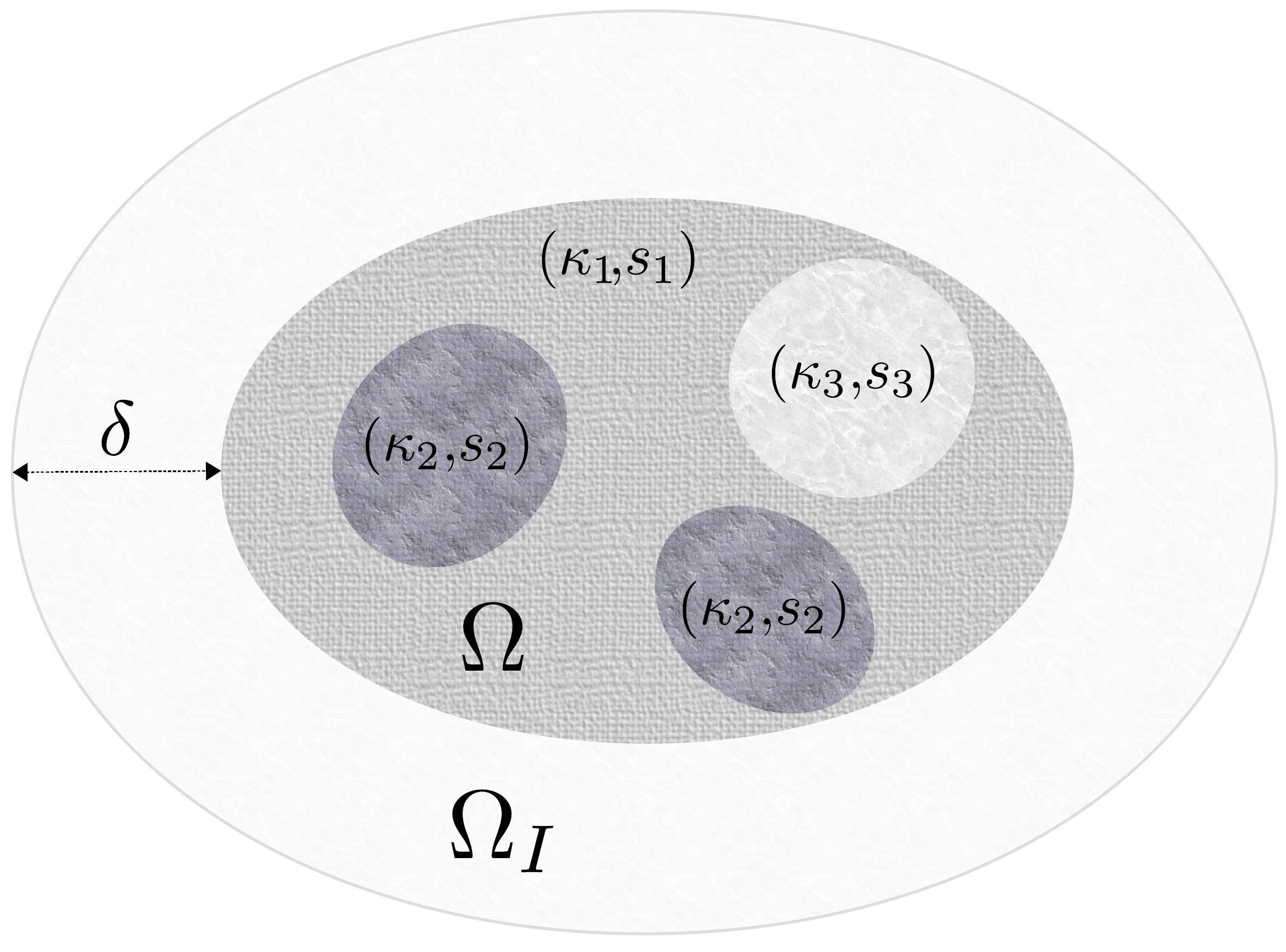}
\hspace{.3cm}
\includegraphics[height=2.in]{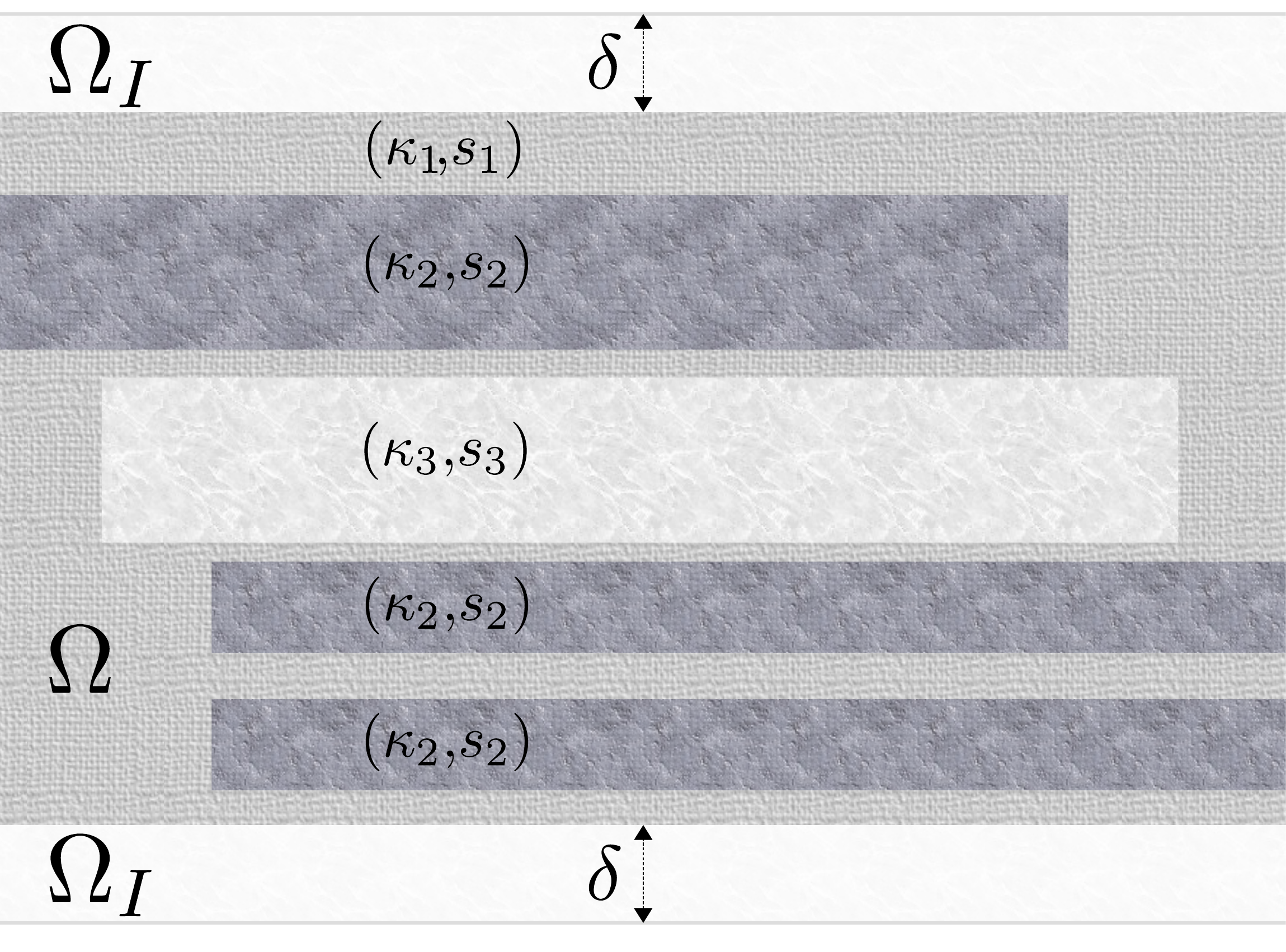}}
\caption{Interface-problem configuration.}
\label{fig:interface-domain-config}
\end{figure}
In this work, we tackle the nonlocal interface problem for fractional-type operators using a generalization of the nonlocal operator in \eqref{eq:variable-truncated-Laplacian}.
Specifically, we add variability to the fractional order and the kernel function itself. Our main contributions are:
\begin{itemize}
\item The introduction of a new variable-order fractional operator characterized by $s=s(\xb,\yb)$ and $\delta\in(0,\infty]$. This choice enables modeling of a much broader set of interface behaviors and, for symmetric $s$, symmetrizes the kernel, making analysis and implementation a much easier task. In fact, note that in \eqref{eq:variable-truncated-Laplacian} the kernel is no longer symmetric, requiring more sophisticated quadrature rules for the integration (regardless of the discretization technique of choice).
\item The analysis of well-posedness of the Poisson problem associated with the new operator in the general nonsymmetric case.
\item The design of an efficient finite-element matrix assembly technique for the practical case of piecewise constant definition of the fractional order that often appears in, e.g., subsurface applications.
\end{itemize}

\paragraph{Outline of the paper} In Section \ref{sec:notation} we first recall some state-of-the-art results; then, we introduce the new variable-order fractional operator and the corresponding strong and weak forms of the nonsymmetric diffusion problem. We also describe properties of the associated energy norm and space. In Section \ref{sec:well-posedness} we report the main result of the paper that proves via Fredholm alternative the well-posedness of the diffusion problem. In Section \ref{sec:FEM} we introduce a finite element discretization and propose efficient quadrature rules for the integration of the nonsymmetric kernel. In Section \ref{sec:numerics} we report several one- and two-dimensional numerical tests that illustrate both the improved variability of our model and the accuracy of the numerical scheme. Finally, in Section \ref{sec:conclusion} we summarize our findings.

\section{A new variable-order fractional model}\label{sec:notation}
In this section we first introduce the notation that will be used throughout the paper and recall state-of-the-art results on nonsymmetric kernels and fractional kernels with variable order. Then, we introduce a new variable-order fractional kernel and discuss the properties of the associated norm and energy space.

\subsection{Preliminaries}
The purpose of this section is two-fold. First, we recall the formulation of a symmetric nonlocal diffusion problem with finite interaction radius following the theory presented in \cite{Du2012,Du2013}; then, we recall the theory presented in \cite{Felsinger2015} on variable-order fractional operators (with infinite interaction radius). Our new model, presented in the next subsection, is a combination of the two.

Let $\omg\subset\mbRn$ be an open bounded domain. We define its {\it interaction domain} as the set of points outside $\omg$ that interact with points inside $\omg$, i.e.
\begin{equation}\label{eq:interaction-domain}
\omgi:=\{\yb\in\mbRn\setminus\omg: |\xb-\yb|\leq\delta, \;\;{\rm for} \; \xb\in\omg\},
\end{equation}
where $\delta\in(0,\infty]$, referred to as interaction radius, determines the extent of the nonlocal interactions. Also, let $\gamma(\xb,\yb):\mbRn\!\times\!\mbRn\to \mbR^+$ be a nonnegative, symmetric, kernel function, not necessarily integrable. The diffusion operator introduced in \cite{Du2013} is defined as
\begin{equation}\label{eq:symm-Laplacian}
-\mcL u(\xb) = {\rm p.v.} \int_\ouo (u(\xb)-u(\yb))\gamma(\xb,\yb)d\yb.
\end{equation}
From now on we remove the explicit reference to the principal value when the operator is used within an equation. The strong form of a truncated, symmetric, diffusion problem is then given by: for $f:\omg\to\mbR$, and $g:\omgi\to\mbR$, find $u:\ouo\to\mbR$ such that
\begin{equation}\label{eq:strong-single-general}
\begin{aligned}
\int_\ouo(u(\xb)-u(\yb))\gamma(\xb,\yb)d\yb = f(\xb) & \quad\xb\in\Omega  \\
u(\xb) = g(\xb) & \quad\xb\in\omgi.
\end{aligned}
\end{equation}
The well-posedness of problem \eqref{eq:strong-single-general} was studied in \cite{Du2012}. Extensions of this model to the nonsymmetric case have been proposed in \cite{Duadvection}, for integrable kernels, and in \cite{DElia2017} for non-integrable kernels by using the operator \eqref{eq:nonlocal-laplacian} introduced in the previous section. This operator is often referred to as nonsymmetric nonlocal diffusion operator or, more appropriately, nonlocal convection-diffusion operator as it introduces a nonlocal convection term. Since the non-symmetry of our problem is not related to convection, but only to the presence of an interface that affects the way a quantity diffuses, in this work, we do not employ \eqref{eq:nonlocal-laplacian}; instead, we use the operator defined in \eqref{eq:symm-Laplacian} with a nonsymmetric kernel $\gamma$.

Among the state-of-the-art formulations of nonlocal diffusion problems with nonsymmetric kernels, we mention the one introduced in \cite{Felsinger2015} for integrable and non-integrable kernels with infinite support. Here the authors consider the nonlocal operator \eqref{eq:symm-Laplacian} where the kernel $\gamma$ is a nonnegative measurable function, possibly nonsymmetric. Among the kernels studied in \cite{Felsinger2015}, we recall the following nonsymmetric variable-order fractional kernel, whose analysis is relevant for the theory presented in this paper:
\begin{equation}\label{eq:Felsinger-kernel}
\gamma_0(\xb,\yb)=\dfrac{\phi_0(\xb)}{|\xb-\yb|^{n+2s(\xb)}},
\end{equation}
where the non-symmetry with respect to $(\xb,\yb)$ is due to the terms $s(\xb)$ and $\phi_0(\xb)$. The latter is a normalization functions with the same role as $C_{n,s}$ in the standard fractional Laplacian in \eqref{eq:fractional-laplacian} \cite{Schilling2015}. Note that the choice of $\gamma_0$ corresponds to the variable-order version of the standard fractional Laplacian operator introduced in \eqref{eq:variable-truncated-Laplacian}. Furthermore, when $s$ and $\phi$ are constant functions, such operator is a multiple of the standard fractional Laplacian. The well-posedness of problem \eqref{eq:strong-single-general} for nonsymmetric kernels was analyzed in \cite{Felsinger2015} in a very general setting and in \cite{Schilling2015} for kernels of the form \eqref{eq:Felsinger-kernel}.

\subsection{The variable-order model}
We the purpose of increasing the descriptive power of the operator $\mcL$, we consider a fractional-type kernel with added variability. 
%
We define the new variable-order nonsymmetric fractional-type kernel as follows:
\begin{align}\label{eq:gamma}
\gamma(\xb,\yb)&= \frac{\phi(\xb,\yb)}{\abs{\xb-\yb}^{n+2s(\xb,\yb)}} \mcX_{\abs{\xb-\yb}\leq\delta}, \qquad \xb,\yb\in\ouo,
\end{align}
with constants \(\underline{s},\overline{s}\), and \(\underline{\phi},\overline{\phi}\) 
such that
\begin{align}
  0<\underline{s}&\leq s(\vec{x},\vec{y})\leq\overline{s}<1, && \forall\;\vec{x},\vec{y}\in\ouo, \label{eq:s-assumptions} \\
  0<\underline{\phi}&\leq\phi(\xb,\yb)\leq \overline{\phi}<\infty && \forall\;\vec{x},\vec{y}\in\ouo. \label{eq:phi-assumptions}
\end{align}
Thus, the resulting nonlocal operator is given by the following expression
\begin{equation}\label{eq:sexy-op}
{-\mcL u(\xb) = {\rm p.v.} \int_{(\Omega\cup\Omega_I)\cap B_\delta(\xb)}(u(\xb)-u(\yb))\dfrac{\phi(\xb,\yb)}{\abs{\xb-\yb}^{n+2s(\xb,\yb)}} d\yb.}
\end{equation}
\smallskip\begin{remark}
As opposed to the kernel model presented in \cite{Schilling2015}, the function $\phi(\xb,\yb)$ does not have a normalizing purpose but depends on the type of phenomenon that we are targeting. As an example, $\phi$ could describe a material property, such as diffusivity or permeability in the subsurface; in this context, the dependence on $\xb$ and $\yb$ describes changes in material properties across material interfaces.
\end{remark}

\smallskip\begin{remark}
For a specific choice of $\phi$, the operator defined in \eqref{eq:sexy-op} corresponds to a variable-order, tempered fractional Laplacian. Specifically, we let
\begin{equation}
\phi(\xb,\yb) = \exp\{-\lambda|\xb-\yb|\},
\end{equation}
where $\lambda>0$ is the so-called tempering parameter, whose effect is to fasten the decay of the fractional kernel. Since the overall effect of tempering is similar to the one of truncation\cite{Olson2020}, in our numerical tests we do not study the tempered case and we limit ourselves to piecewise constant definitions of $\phi$ for $\delta\in(0,\infty]$.
\end{remark}

\smallskip\begin{remark}
When $s(\xb,\yb)=s(\yb,\xb)$ the kernel $\gamma$ is  symmetric. This case is much easier to deal with both in terms of analysis and computation. In fact, in the symmetric case, the theory introduced in \cite{Du2012} can be readily applied and guarantees well-posedness of problem \eqref{eq:strong-single-general}. Furthermore, numerical integration of a symmetric function poses minor challenges compared with the nonsymmetric case, which requires the employment of more sophisticated quadrature rule. As we show in our numerical tests, a symmetric $s$ still guarantees improved model variability across interfaces, compared to the single-variable model in \eqref{eq:Felsinger-kernel}.
\end{remark}

\smallskip We define the spaces
\begin{align*}
H(\ouo;\gamma)&:=\left\{v: \ouo \rightarrow \mathbb{R}: v\in L^{2}(\ouo), \opnorm{v}<\infty \right\}, \\
\V&:=H_{\Omega}(\ouo;\gamma):=\left\{v\in H(\ouo;\gamma) : v|_{\omgi}=0 \right\},
\end{align*}
where the semi-norm $\opnorm{v}^{2}$ is defined as
\begin{align*}
\opnorm{v}^{2} := \iint_{(\ouo)^{2}} (v(\vec{x})-v(\vec{y}))^{2}\gamma_{s}(\vec{x},\vec{y})\,d\vec{x}\,d\vec{y}.
\end{align*}
Here, $\gamma_s$ is the symmetric part of the kernel. We recall that the symmetric and anti-symmetric parts of the kernel $\gamma$ are defined as follows:
\begin{align}\label{eq:gammas-gammaa}
   \gamma_{s}(\vec{x},\vec{y}) &= \frac{1}{2}\left(\dfrac{\phi(\vec{x},\vec{y})}{|\vec{x}-\vec{y}|^{n+2s(\vec{x},\vec{y})}}
+ \dfrac{\phi(\vec{y},\vec{x})}{|\vec{x}-\vec{y}|^{n+2s(\vec{y},\vec{x})}}\right)\mathcal{X}_{\abs{\vec{x}-\vec{y}}\leq\delta}, \\[3mm]
   \gamma_{a}(\vec{x},\vec{y}) &= \frac{1}{2}\left(\dfrac{\phi(\vec{x},\vec{y})}{|\vec{x}-\vec{y}|^{n+2s(\vec{x},\vec{y})}}
-  \dfrac{\phi(\vec{y},\vec{x})}{|\vec{x}-\vec{y}|^{n+2s(\vec{y},\vec{x})}}\right)\mathcal{X}_{\abs{\vec{x}-\vec{y}}\leq\delta}.
\end{align}
The following bounds hold trivially:
\begin{align}
  \frac{1}{\abs{\xb-\yb}^{d+2\underline{s}}}&\leq \frac{1}{\abs{\xb-\yb}^{d+2s(\xb,\yb)}} \leq \frac{1}{\abs{\xb-\yb}^{d+2\overline{s}}} &\text{if }\abs{\xb-\yb}\leq 1, \label{eq:kernelBoundsNear} \\
  \frac{1}{\abs{\xb-\yb}^{d+2\overline{s}}}&\leq \frac{1}{\abs{\xb-\yb}^{d+2s(\xb,\yb)}} \leq \frac{1}{\abs{\xb-\yb}^{d+2\underline{s}}} &\text{if }\abs{\xb-\yb}\geq 1, \label{eq:kernelBoundsFar}
\end{align}
In what follows, we will also frequently make use of the following lower bound on the horizon:
\begin{align}\label{eq:lowerBoundHorizon}
  \underline{\delta}:=\min\{\delta,1\}.
\end{align}
The following Lemma shows that the semi-norm above satisfies a Poincar\'e inequality and, hence, equips $\V$ with a norm.

\begin{lemma}
For all $u\in \V$, the following Poincar\'e inequality holds:
\begin{align*}
C_{p}\opnorm{u}\geq \norm{u}_{L^{2}(\ouo)}.
\end{align*}
\end{lemma}
\begin{proof}
Assumptions \eqref{eq:s-assumptions} and inequality \eqref{eq:kernelBoundsNear} imply
\begin{align*}
\opnorm{v}^{2}
&= \iint_{(\ouo)^{2},\abs{\vec{x}-\vec{y}}\leq \underline{\delta}}
   (v(\vec{x})-v(\vec{y}))^{2}\gamma_{s}(\vec{x},\vec{y})\,d\vec{x}\,d\vec{y}\\
&+ \iint_{(\ouo)^{2},\abs{\vec{x}-\vec{y}}\geq \underline{\delta}}
   (v(\vec{x})-v(\vec{y}))^{2}\gamma_{s}(\vec{x},\vec{y})\,d\vec{x}\,d\vec{y} \\
&\geq C\iint_{(\ouo)^{2},\abs{\vec{x}-\vec{y}}\leq\underline{\delta}}
   (v(\vec{x})-v(\vec{y}))^{2}\frac{1}{\abs{\vec{x}-\vec{y}}^{n+2\underline{s}}}\,d\vec{x}\,d\vec{y}.
\end{align*}
As a consequence, $\opnorm{\cdot}$ is lower bounded by the semi-norm of a $\underline{\delta}$-truncated fractional Sobolev space with fractional order $\underline s$, for which the Poincar\'e inequality holds, see \cite[Lemma 4.3]{Du2012}.
\end{proof}

\smallskip\noindent\textbf{The weak form.}
Let $g=0$ for simplicity; then, we can write the weak form of problem \eqref{eq:strong-single-general} as
\begin{equation}\label{eq:weak-form}
\int_\omg -\mcL u \,v \,d\xb = 
{\iint_{(\ouo)^2} (u(\xb)-u(\yb))\gamma(\xb,\yb)v(\xb)\,d\yb\,d\xb}
= \int_\omg f\,v\,d\xb,
\end{equation}
We denote the bilinear form in \eqref{eq:weak-form} by
\begin{equation}\label{eq:A}
{\mcA(u,v)= \iint_{(\ouo)^2} \gamma(\xb,\yb)v(\xb)(u(\xb)-u(\yb)\,d\yb\,d\xb.}
\end{equation}
While this form is formally equivalent to the one introduced in \cite{Felsinger2015}, the kernel \eqref{eq:gamma} does not belong to class of kernels studied in that paper. In the next section we show how to extend the well-posedness theory developed in \cite{Felsinger2015} and \cite{Schilling2015} to our new operator.

For discretization purposes, it is convenient to further rewrite \(\mcA\) as
\begin{equation}\label{eq:A2}
    \mcA(u,v)= \frac{1}{2}\iint_{(\ouo)^2} (u(\xb)-u(\yb)) \left(\gamma(\xb,\yb)v(\xb) - \gamma(\yb,\xb)v(\yb)\right) \,d\yb\,d\xb.
\end{equation}
This form allows for simpler numerical integration and is used in our numerical implementation.

\section{Well-posedness analysis}\label{sec:well-posedness}
In this section we prove the well-posedness of problem \eqref{eq:weak-form} by adapting the theory developed in the works of Felsinger, Kassmann and Voigt~\cite{Felsinger2015} and of Schilling and Wang~\cite{Schilling2015}. We first report the building blocks of our main well-posedness result.

The following proposition establishes a relation between the symmetric and anti-symmetric part of the kernel and is a generalization of Proposition 3.1 in the paper by Schilling and Wang~\cite{Schilling2015}.
\begin{proposition}\label{prop:Agamma}
Assume that
\begin{align}\label{eq:beta-property}
  \beta(r):=\sup_{\abs{\vec{x}-\vec{y}}\leq r}\abs{s(\vec{x},\vec{y})-s(\vec{y},\vec{x})}
\end{align}
satisfies
\begin{align}\label{as:integrability}
  \int_{0}^{1}dr~\frac{\left(\beta(r)\abs{\log r}\right)^{2}}{r^{1+2\overline{s}}} <\infty.
\end{align}
Moreover, assume that
\begin{align*}
  \alpha(r):=\sup_{\abs{\vec{x}-\vec{y}}\leq r}\abs{\phi(\vec{x},\vec{y})-\phi(\vec{y},\vec{x})}
\end{align*}
is \((\overline{s}+\varepsilon)\)-H\"older continuous for \(\varepsilon>0\) arbitrarily small.
Then there exists a constant \(A_{\gamma}\) such that
\begin{align}\label{eq:4}
  \sup_{\vec{x}\in\ouo}\int_{\{\gamma_{s}(\xb,\yb)\neq0\}} d\vec{y}~\frac{\gamma_{a}(\vec{x},\vec{y})^{2}}{\gamma_{s}(\vec{x},\vec{y})} \leq A_{\gamma}< \infty.
\end{align}
\end{proposition}
\begin{proof}
First, by definition of symmetric and anti-symmetric parts of the kernel in \eqref{eq:gammas-gammaa}, we have that \(\abs{\gamma_{a}(\vec{x},\vec{y})}\leq \gamma_{s}(\vec{x},\vec{y})\).
This and assumptions \eqref{eq:s-assumptions}, \eqref{eq:phi-assumptions} and inequality \eqref{eq:kernelBoundsFar} imply that there exists a positive constant \(C_{1}\) such that
\begin{align*}
    \int_{\{\gamma_{s}(\xb,\yb)\neq0\},\abs{\vec{x}-\vec{y}}\geq 1} d\vec{y}~\frac{\gamma_{a}(\vec{x},\vec{y})^{2}}{\gamma_{s}(\vec{x},\vec{y})}
    \leq \int_{\abs{\vec{x}-\vec{y}}\geq 1} d\vec{y}~\gamma_{s}(\vec{x},\vec{y})
    \leq C \int_{1}^{\infty} r^{-1-2\underline{s}}
    =: C_{1}.
\end{align*}
Moreover, due to the definition of \(\underline{\delta}\),
\begin{align*}
  \int_{\{\gamma_{s}(\xb,\yb)\neq0\},\abs{\vec{x}-\vec{y}}\in [\underline{\delta},1]} d\vec{y}~\frac{\gamma_{a}(\vec{x},\vec{y})^{2}}{\gamma_{s}(\vec{x},\vec{y})} =0
\end{align*}
so that we obtain
\begin{align}\label{eq:partI}
  \int_{\{\gamma_{s}(\xb,\yb)\neq0\},\abs{\vec{x}-\vec{y}}\geq \underline{\delta}} d\vec{y}~\frac{\gamma_{a}(\vec{x},\vec{y})^{2}}{\gamma_{s}(\vec{x},\vec{y})} \leq C_{1}
\end{align}
Next, for \(\abs{\vec{x}-\vec{y}}\leq \underline{\delta}\), we write:
\begin{align*}
    \gamma_{a}(\vec{x},\vec{y})
    &= \frac{1}{2}\left(\phi(\vec{x},\vec{y})-\phi(\vec{y},\vec{x})\right)\abs{\vec{x}-\vec{y}}^{-n-2s(\vec{x},\vec{y})} \\
    &\quad+ \frac{1}{2}\phi(\vec{y},\vec{x})\abs{\vec{x}-\vec{y}}^{-n}\left(\abs{\vec{x}-\vec{y}}^{-2s(\vec{x},\vec{y})}-\abs{\vec{x}-\vec{y}}^{-2s(\vec{y},\vec{x})}\right). \\
    &=: \gamma_{a,1}(\vec{x},\vec{y}) + \gamma_{a,2}(\vec{x},\vec{y})
\end{align*}
Then, by \eqref{eq:phi-assumptions} and the H\"older continuity of \(\alpha\) we have that there exists a positive constant \(C_{2}\) such that
\begin{align}
    &\int_{\{\gamma_{s}(\xb,\yb)\neq0\},\abs{\vec{x}-\vec{y}}\leq \underline{\delta}} d\vec{y}~\frac{\gamma_{a,1}(\vec{x},\vec{y})^{2}}{\gamma_{s}(\vec{x},\vec{y})} \nonumber \\
    &=\frac{1}{2}\int_{\{\gamma_{s}(\xb,\yb)\neq0\},\abs{\vec{x}-\vec{y}}\leq \underline{\delta}} d\vec{y}~\frac{\left(\phi(\vec{x},\vec{y})-\phi(\vec{y},\vec{x})\right)^{2}\abs{\vec{x}-\vec{y}}^{-2n-4s(\vec{x},\vec{y})}}{\phi(\vec{x},\vec{y})\abs{\vec{x}-\vec{y}}^{-n-2s(\vec{x},\vec{y})}+\phi(\vec{y},\vec{x})\abs{\vec{x}-\vec{y}}^{-n-2s(\vec{y},\vec{x})}} \nonumber \\
    &\leq C\int_{\{\gamma_{s}(\xb,\yb)\neq0\},\abs{\vec{x}-\vec{y}}\leq \underline{\delta}} d\vec{y}~\frac{\left(\phi(\vec{x},\vec{y})-\phi(\vec{y},\vec{x})\right)^{2}}{\abs{\vec{x}-\vec{y}}^{n+2s(\vec{x},\vec{y})}} \nonumber \\
    &\leq C\int_{\{\gamma_{s}(\xb,\yb)\neq0\},\abs{\vec{x}-\vec{y}}\leq \underline{\delta}} d\vec{y}~ \frac{\alpha(\abs{\vec{x}-\vec{y}})^{2}}{\abs{\vec{x}-\vec{y}}^{n+2\overline{s}}} \nonumber \\
    &\leq C\int_{0}^{\underline{\delta}} dr~ \frac{r^{2\overline{s}+2\varepsilon}}{r^{1+2\overline{s}}} =: C_{2}. \label{eq:partII}
\end{align}
Here, we have used \eqref{eq:kernelBoundsNear}.
On the other hand, since for \(a,b>0\)
\begin{align*}
    \abs{\vec{x}-\vec{y}}^{-a}-\abs{\vec{x}-\vec{y}}^{-b} &= -\int_{a}^{b}du~\abs{\vec{x}-\vec{y}}^{-u}\log\abs{\vec{x}-\vec{y}},
\end{align*}
then, for \(\abs{\vec{x}-\vec{y}}\leq\underline{\delta}\leq1\), we have that
\begin{align*}
    \left(\abs{\vec{x}-\vec{y}}^{-2s(\vec{x},\vec{y})}-\abs{\vec{x}-\vec{y}}^{-2s(\vec{y},\vec{x})}\right)^{2}
    &\leq \left(\log\abs{\vec{x}-\vec{y}}\right)^{2}\left(s(\vec{x},\vec{y})-s(\vec{y},\vec{x})\right)^{2}\abs{\vec{x}-\vec{y}}^{-4\min\{s(\vec{x},\vec{y}),s(\vec{y},\vec{x})\}}.
\end{align*}
Hence, by \eqref{eq:phi-assumptions} and assumption \eqref{as:integrability}, there exists a positive constant \(C_{3}\) such that
\begin{align}
    &\int_{\{\gamma_{s}(\xb,\yb)\neq0\},\abs{\vec{x}-\vec{y}}\leq \underline{\delta}} d\vec{y}~\frac{\gamma_{a,2}(\vec{x},\vec{y})^{2}}{\gamma_{s}(\vec{x},\vec{y})} \nonumber \\
    =&\frac{1}{2}\int_{\{\gamma_{s}(\xb,\yb)\neq0\},\abs{\vec{x}-\vec{y}}\leq \underline{\delta}} d\vec{y}~\frac{\phi(\vec{y},\vec{x})^{2}\abs{\vec{x}-\vec{y}}^{-2n}\left(\abs{\vec{x}-\vec{y}}^{-2s(\vec{x},\vec{y})}-\abs{\vec{x}-\vec{y}}^{-2s(\vec{y},\vec{x})}\right)^{2}}{\phi(\vec{x},\vec{y})\abs{\vec{x}-\vec{y}}^{-n-2s(\vec{x},\vec{y})}+\phi(\vec{y},\vec{x})\abs{\vec{x}-\vec{y}}^{-n-2s(\vec{y},\vec{x})}} \nonumber \\
    \leq& C \int_{\{\gamma_{s}(\xb,\yb)\neq0\},\abs{\vec{x}-\vec{y}}\leq \underline{\delta}} d\vec{y}~\frac{\left(\log\abs{\vec{x}-\vec{y}}\right)^{2}\left(s(\vec{x},\vec{y})-s(\vec{y},\vec{x})\right)^{2}}{\abs{\vec{x}-\vec{y}}^{2n+4\min\{s(\vec{x},\vec{y}),s(\vec{y},\vec{x})\}}}
           \frac{1}{\phi(\vec{x},\vec{y})\abs{\vec{x}-\vec{y}}^{-n-2s(\vec{x},\vec{y})}+\phi(\vec{y},\vec{x})\abs{\vec{x}-\vec{y}}^{-n-2s(\vec{y},\vec{x})}} \nonumber \\
    \leq& C \int_{\{\gamma_{s}(\xb,\yb)\neq0\},\abs{\vec{x}-\vec{y}}\leq \underline{\delta}} d\vec{y}~ \frac{\left(\log\abs{\vec{x}-\vec{y}}\right)^{2}\left(s(\vec{x},\vec{y})-s(\vec{y},\vec{x})\right)^{2}}{\abs{\vec{x}-\vec{y}}^{n+2\overline{s}}} \nonumber \\
    \leq& C \int_{0}^{\underline{\delta}} dr~ \frac{\left(\beta(r)\abs{\log r}\right)^{2}}{r^{1+2\overline{s}}} =: C_{3}. \label{eq:partIII}
\end{align}
Here, we have again used \eqref{eq:kernelBoundsNear}.
Hence, from \eqref{eq:partI}, \eqref{eq:partII} and \eqref{eq:partIII}, we obtain:
\begin{align*}
    \sup_{\vec{x}\in\ouo}\int_{\ouo} d\vec{y}~\frac{\gamma_{a}(\vec{x},\vec{y})^{2}}{\gamma_{s}(\vec{x},\vec{y})}
    &\leq C_{1}+2(C_{2}+C_{3}) =:A_{\gamma}.
\end{align*}
and the result follows.
\end{proof}

The next G{\aa}rding inequality follows directly from \cite[Lemma 3.1]{Felsinger2015}; we report its proof for completeness.
\begin{lemma}[G{\aa}rding inequality]
There exist constants \(C,c>0\) such that
\begin{align}\label{eq:Garding}
\mathcal{A}(u,u) \geq C\opnorm{u}^{2} -c\norm{u}_{L^{2}(\ouo)}^{2} \qquad \forall \;u\in \V.
\end{align}
\end{lemma}
\begin{proof}
By using Young's inequality for \(\varepsilon>0\), we obtain
  \begin{align*}
    \mathcal{A}(u,u)
    &= \iint\limits_{\left(\ouo\right)^{2}} \gamma(\vec{x},\vec{y})u(\vec{x})(u(\vec{x})-u(\vec{y}))\,d\vec{y}\,d\vec{x} \\
    & \geq \frac{1}{2}\iint\limits_{\left(\ouo\right)^{2}} \gamma_{s}(\vec{x},\vec{y})\abs{u(\vec{x})-u(\vec{y})}^{2}\,d\vec{y}\,d\vec{x}
    - \iint\limits_{\left(\ouo\right)^{2}}  \abs{\gamma_{a}(\vec{x},\vec{y})u(\vec{x})(u(\vec{x})-u(\vec{y}))}\,d\vec{y}\,d\vec{x} \\
    &= \frac{1}{2} \opnorm{u}^{2} - \iint\limits_{\left(\ouo\right)^{2}} \abs{\gamma_{a}(\vec{x},\vec{y})\gamma_{s}(\vec{x},\vec{y})^{1/2}\gamma_{s}(\vec{x},\vec{y})^{-1/2}u(\vec{x})(u(\vec{x})-u(\vec{y}))}\,d\vec{y}\,d\vec{x} \\
    &\geq \frac{1}{2} \opnorm{u}^{} - \varepsilon\iint\limits_{\left(\ouo\right)^{2}} \gamma_{s}(\vec{x},\vec{y})\abs{(u(\vec{x})-u(\vec{y}))}^{2}\,d\vec{y}\,d\vec{x} - \frac{1}{4\varepsilon}\iint\limits_{\{\gamma_{s}(\xb,\yb)\neq0\}^{2}} \frac{\gamma_{a}(\vec{x},\vec{y})^{2}}{\gamma_{s}(\vec{x},\vec{y})}u(\vec{x})^{2}\,d\vec{y}\,d\vec{x} \\
    &\geq \left(\frac{1}{2}-\varepsilon\right) \opnorm{u}^{2} - \frac{A_{\gamma}}{4\varepsilon}\norm{u}_{L^{2}(\ouo)}^{2}.
  \end{align*}
  Hence, for \(\varepsilon\) small enough, we obtain
  \begin{align*}
    \mathcal{A}(u,u) \geq C\opnorm{u}^{2} -c\norm{u}_{L^{2}(\ouo)}^{2}.
  \end{align*}
\end{proof}

The Poincar\'e inequality and Proposition \ref{prop:Agamma} yield the continuity of the bilinear form $\mcA$, as illustrated in the following lemma. 
\begin{lemma}\label{lem:Acontinuity}
For all \(u,v\in \V\) 
\begin{equation*}
\abs{\mathcal{A}(u,v)} \leq C(C_{P},A_{\gamma})\opnorm{u}\opnorm{v},
\end{equation*}
i.e., the bilinear form $\mcA(\cdot,\cdot)$ is continuous on \(\V\times \V\).
\end{lemma}
\begin{proof} 
By combining the Poincar\'e inequality and Proposition  \ref{prop:Agamma}, we have
\begin{align*}
\abs{\mathcal{A}(u,v)}
&= \abs{\;\;\iint\limits_{\left(\ouo\right)^{2}} \left(\gamma_{s}(\vec{x},\vec{y})+\gamma_{a}(\vec{x},\vec{y})\right)(u(\vec{x})-u(\vec{y}))v(\vec{x})\,d\vec{y}\,d\vec{x}\,} \\
&\leq \opnorm{u}\opnorm{v}+\left(\;\;\iint\limits_{\left(\ouo\right)^{2}} \gamma_{s}(\vec{x},\vec{y})(u(\vec{x})-u(\vec{y}))^{2}\,d\vec{y}\,d\vec{x}\right)^{1/2}\left(\;\;\iint\limits_{\{\gamma_{s}(\xb,\yb)\neq0\}^{2}} \frac{\gamma_{a}(\vec{x},\vec{y})}{\gamma_{s}(\vec{x},\vec{y})}v(\vec{x})^{2}\,d\vec{y}\,d\vec{x}\right)^{1/2}  \\
  &\leq \opnorm{u}\opnorm{v} + A_{\gamma}^{1/2} \opnorm{u} \norm{v}_{L^{2}({\ouo})} \\[2mm]
  & \leq (1+C_{P}A_{\gamma}^{1/2})\opnorm{u}\opnorm{v}.
\end{align*}
\end{proof}

Before proceeding with the well-posedness analysis, we prove a result for a perturbation of the weak problem \eqref{eq:weak-form}. First, let
\begin{equation*}
  \underline{\gamma}(\vec{x},\vec{y}):= \frac{1}{\abs{\vec{x}-\vec{y}}^{n+2\underline{s}}}\mathcal{X}_{\abs{\vec{x}-\vec{y}}\leq \underline{\delta}},
  \quad {\rm and} \quad 
  \tilde{H}_{\underline{\delta}}^{\underline{s}}(\Omega) := H_{\Omega}(\ouo;\underline{\gamma}).
\end{equation*}
Then, there exists some \(\lambda>0\) such that
\begin{align*}
  \gamma_{s}(\vec{x},\vec{y})\geq \lambda \underline{\gamma}(\vec{x},\vec{y}),
\end{align*}
and hence for all $u\in\V$,
\begin{align}\label{eq:E-lambda}
  \opnorm{u}^{2}=\iint\limits_{(\ouo)^{2}} (u(\vec{x})-u(\vec{y}))^{2}\gamma_{s}(\vec{x},\vec{y})\,d\vec{y}\,d\vec{x}
  &\geq \lambda \iint\limits_{(\ouo)^{2}} (u(\vec{x})-u(\vec{y}))^{2}\underline{\gamma}(\vec{x},\vec{y})\,d\vec{y}\,d\vec{x} = \opnorm[\underline{\gamma}]{u}^{2}.
\end{align}
Hence, \(\V\) is continuously embedded in \(\tilde{H}_{\underline{\delta}}^{\underline{s}}(\Omega)\).
Since \(\tilde{H}_{\underline{\delta}}^{\underline{s}}(\Omega)\) is compactly embedded in \(L^{2}(\Omega)\), and \(L^{2}(\Omega)\) is compactly embedded in \(\V^{*}\), \(\left(\V,L^{2}(\Omega),\V^{*}\right)\) forms a Gelfand triple, and the embedding of \(\V\) into \(\V^{*}\) is compact.
This fact, and the G{\aa}rding inequality \eqref{eq:Garding}, directly imply the well-posedness of a perturbation of problem \eqref{eq:weak-form} and a bound on its solution, as shown in the following lemma.
In what follows, we denote the duality pairing between $\V$ and its dual in the usual manner, i.e. $\langle \cdot,\cdot\rangle$.
\begin{lemma}\label{lem:perturbed-problem}
Given a positive constant $c$ and a function $f\in \V^{*}$, the problem 
\begin{align*}
\hbox{Find $u\in \V$ such that} \;\;
\mathcal{A}(u,v)+c(u,v)_{L^{2}(\omg)} &= \left\langle f,v\right\rangle, 
\;\; \forall\;v\in \V, \;\;\hbox{and $u=0$ in $\omgi$}
\end{align*}
is well-posed and its solution is such that \(\opnorm{u}\leq C\norm{f}_{\V^{*}}\), for a positive constant $C$.
\end{lemma}

\smallskip The proof is simply based on application of \eqref{eq:Garding} to the perturbed bilinear form, and is hence omitted. Note that \eqref{eq:E-lambda} also allows us to apply \cite[Theorem 4.1]{Felsinger2015}, which states that the bilinear form $\mcA$ satisfies a weak maximum principle, which we report for completeness.
\begin{lemma}
Let $\gamma$ and $\mcA$ be defined as in \eqref{eq:Felsinger-kernel} and \eqref{eq:A} and let $u\in\V$ satisfy
\begin{equation*}
\mcA(u,v)\leq 0 \quad \forall \;v\in \V,    
\end{equation*}
then, $\sup_\omg u \leq 0$. In particular, 
\begin{equation}\label{eq:trivial-solution}
\mcA(u,v) = 0 \;\; \forall \;v\in \V \qquad
\Longleftrightarrow \qquad u=0.
\end{equation}
\end{lemma}

An immediate consequence of Lemma \ref{lem:perturbed-problem} is the fact that the operator $\mcK:=(\mcL+cI)^{-1}$ mapping \(\V^{*}\rightarrow \V\), is a compact operator. This fact allows us to use the Fredholm alternative theorem to prove the well-posedness of problem \eqref{eq:weak-form}.
\smallskip
\begin{theorem}\label{thm:Fredholm}
Given a function $f\in \V^{*}$, the problem 
\begin{align*}
\hbox{Find $u\in \V$ such that} \;\;
\mathcal{A}(u,v)= \left\langle f,v\right\rangle, 
\;\; \forall\;v\in \V, \;\;\hbox{and $u=0$ in $\omgi$}
\end{align*}
has a unique solution $u\in\V$.
\end{theorem}
\begin{proof}
Consider the operator $\mcK-\eta I$. We have the following identity: 
\begin{align*}
  \mcK-\eta I = -\eta \mcK (\mcK^{-1}-\eta^{-1}I) = -\eta \mcK(\mcL+(c-\eta^{-1})I).
\end{align*}
By the Fredholm alternative, either \(\eta\) is an eigenvalue of \(\mcK\), or \((\mcK-\eta I)^{-1}\) is a bounded linear operator. We show that the first option leads to a contradiction.

Let \(\eta=1/c\) and assume that \(\eta\) is an eigenvalue of \(\mcK\). This implies that $\mcA(u,v)=0$ for all $v\in \V$ and, by \eqref{eq:trivial-solution}, that $u=0$. Hence, \(\eta\) is not an eigenvalue of \(\mcK\). This ends the proof since we can conclude that \((\mcK-\eta I)^{-1}\), and subsequently \((-\eta \mcK-\eta I)^{-1}\mcK=\mcL^{-1}\), are bounded linear operators. In other words, the variational problem associated with the operator $\mcL$ is well-posed. 
\end{proof}

\section{Finite element discretization and efficient matrix assembly}\label{sec:FEM}
In this section we briefly introduce the finite element (FE) discretization of problem \eqref{eq:weak-form} and provide details regarding numerical integration.

We choose a family of finite-dimensional subspaces $V_{h} \subset \V$, parameterized by the mesh size $h$.
We define the Galerkin approximation $u_h\in V_{h}$ as the solution of the following problem: find $u_h\in V_{h}$ such that
\begin{equation}\label{eq:Galerkin-projetion}
  \mcA(u_{h},v_{h})= \int\limits_\omg f\,v_h\,d\xb, \quad \forall\,v_h\in V_{h}.
\end{equation}
As in \cite{Du2012}, we consider finite element approximations for the case that both $\omg$ and $\omgi$ are polyhedral domains.
We partition $\omg$ into finite elements and denote by $h$ the diameter of the largest element in the partition. We assume that the partition is shape-regular and quasi-uniform \cite{Brenner2008} as the grid size $h\to 0$ and we choose the subspace $V^h$ to consist of piecewise polynomials of degree no more than $m$ defined with respect to the partition.

We perform numerical integration on pairs of elements, by rewriting the bilinear form \eqref{eq:A2} as the sum of integrals over the partition, i.e.
\begin{align*}
  \mcA(u_{h},v_{h})= \frac{1}{2} \sum_{K}\sum_{\tilde{K}} \iint_{K\times\tilde{K}} (u_{h}(\xb)-u_{h}(\yb)) \left(\gamma(\xb,\yb)v_{h}(\xb) - \gamma(\yb,\xb)v_{h}(\yb)\right) \,d\yb\,d\xb.
\end{align*}

The design of quadrature rules for the integrals above depends on the regularity properties of the functions $s$ and $\phi$.
If \(s\) and \(\phi\) are constant on a pair of elements \(K\times\tilde{K}\), we can use the state-of-the-art quadrature rules described in \cite{AinsworthGlusa2018,AinsworthGlusa2017} designed for the constant-order fractional Laplacian, i.e. for the case $\delta=\infty$.
However, when this is not the case, the use of more sophisticated and often more expensive quadrature rules, such as adaptive quadrature rules \cite{PiessensDoncker-KapengaEtAl2012_Quadpack}, becomes mandatory to prevent the integration error to be dominant.
Conditions for well-posedness derived in Proposition~\ref{prop:Agamma} are such that any non-symmetric kernel falls into the latter category, hence requiring higher computational effort.
In addition, note that, as for the constant-order fractional Laplacian, when \(K\cap\tilde{K}\neq\emptyset\), the integrand function features a singularity at  \(\xb=\yb\). In this case, the domain of integration must be partitioned into subdomains in order to avoid the singularity.

In the specific case of piecewise constant fractional order, we speed up the assembly of the discrete operator and subsequent matrix-vector multiplications by generalizing the panel-clustering approach of \cite{AinsworthGlusa2018,AinsworthGlusa2017} developed for the constant-order fractional Laplacian.
Informally speaking, in the standard version of this approach, the unknowns are recursively grouped into nodes of a tree structure, with the root node containing all unknowns.
Pairs of clusters correspond to sub-blocks of the discrete operator.
If certain conditions on their location in physical space are satisfied, sub-blocks are approximated using Chebyshev interpolation, otherwise they are assembled using quadrature\footnote{See the works by Ainsworth and Glusa\cite{AinsworthGlusa2018,AinsworthGlusa2017} for a detailed description of this approach for constant parameters and infinite horizon.}.
We generalize this algorithm to finite horizon and piecewise constant coefficients by restricting the approximation of matrix sub-blocks to cluster pairs that are strictly contained within the horizon \(\delta\) of each other and whose interaction reduces to a constant fractional kernel.

\section{Numerical tests}\label{sec:numerics}

\subsection{Comparison of different types of kernels}

In order to highlight the different behaviors of available fractional kernels, we restrict ourselves to a one-dimensional problem where model parameters are constant over two subdomains.
This configuration can be considered a one-dimensional counterpart of the two-dimensional configurations displayed in Figure \ref{fig:interface-domain-config}.
Specifically, we let \(\omg=(-1,1)\), \(\delta=1\), and consider kernels given by \eqref{eq:gamma} that are piecewise constant with respect to the partition \(\ouo=((\ouo)\cap\mbR_{-})\cup((\ouo)\cap\mbR_{+})\).
In order for the assumptions of Proposition~\ref{prop:Agamma} to be satisfied, such kernels need to be symmetric, i.e. \(s\) and \(\phi\) need to take the form
\begin{align*}
  \psi(x,y;\psi_{+},\psi_{-},\psi_{\pm})
  &= \begin{cases}
    \psi_{+} & \text{if } x,y>0,\\
    \psi_{-} & \text{if } x,y<0, \\
    \psi_{\pm} & \text{else}.
  \end{cases}
\end{align*}
Note that this excludes the interesting cases, that can be found in the literature, of piecewise constant \(\phi\) or \(s\) that only depend on \(x\).
Hence, we also consider functions of the following form and compare them with the proposed model.
\begin{align*}
  \eta(x;\eta_{+},\eta_{-},r,\kappa)
  &=
    \begin{cases}
      \eta_{+} & \text{if } x\geq r \\
      \frac{1}{2}(\eta_{+}+\eta_{-})+\frac{1}{2}(\eta_{+}-\eta_{-}) \frac{\arctan{\kappa\xb}}{\arctan{\kappa r}} & \text{if } -r\leq x\leq r\\
      \eta_{-} & \text{if } x\leq -r.
    \end{cases}
\end{align*}
Here, \(r\) determines the size of the transition region between the two states \(\eta_{+}\) and \(\eta_{-}\), and \(\kappa\) the steepness of the transition.
Since \(\eta\) is Lipschitz continuous, we can use \(\eta\) for both \(s\) and \(\phi\) and satisfy the conditions of Proposition~\ref{prop:Agamma}.

We numerically solve the diffusion problem \eqref{eq:strong-single-general} for homogeneous Dirichlet data \(g(x)=0\) and forcing term \(f(x)=\mcX_{\abs{x-0.4}<0.2}\).
In all tests conducted in this section we use continuous piecewise linear FE on a uniform mesh of size $h=2^{-10}$.

With the purpose of highlighting how a variable fractional order affects the qualitative behavior of the solution across the interface, we first compare the family of symmetric kernels reported in Table~\ref{tab:sym} and display the corresponding solutions in Figure~\ref{fig:sym-solutions}.
We observe that the solutions for constant \(s\) seem to be differentiable, whereas the cases with piecewise defined \(s\) are only continuous at \(x=0\).
\begin{table}
  \centering
  \begin{tabular}{lll}
    identifier & \(s\) & \(\phi\) \\
    \hline
    const, const & \(0.75\) & \(1\) \\
    const, sym & \(0.75\) & \(\psi(\cdot;1,1,0.1)\) \\
    sym, const & \(\psi(\cdot;0.75,0.25,0.5)\) & \(1\) \\
    sym, sym & \(\psi(\cdot;0.75,0.25,0.5)\) & \(\psi(\cdot;1,1,0.1)\) \\
  \end{tabular}
  \caption{Symmetric kernels}
  \label{tab:sym}
\end{table}
\begin{figure}
  \centering
  \includegraphics{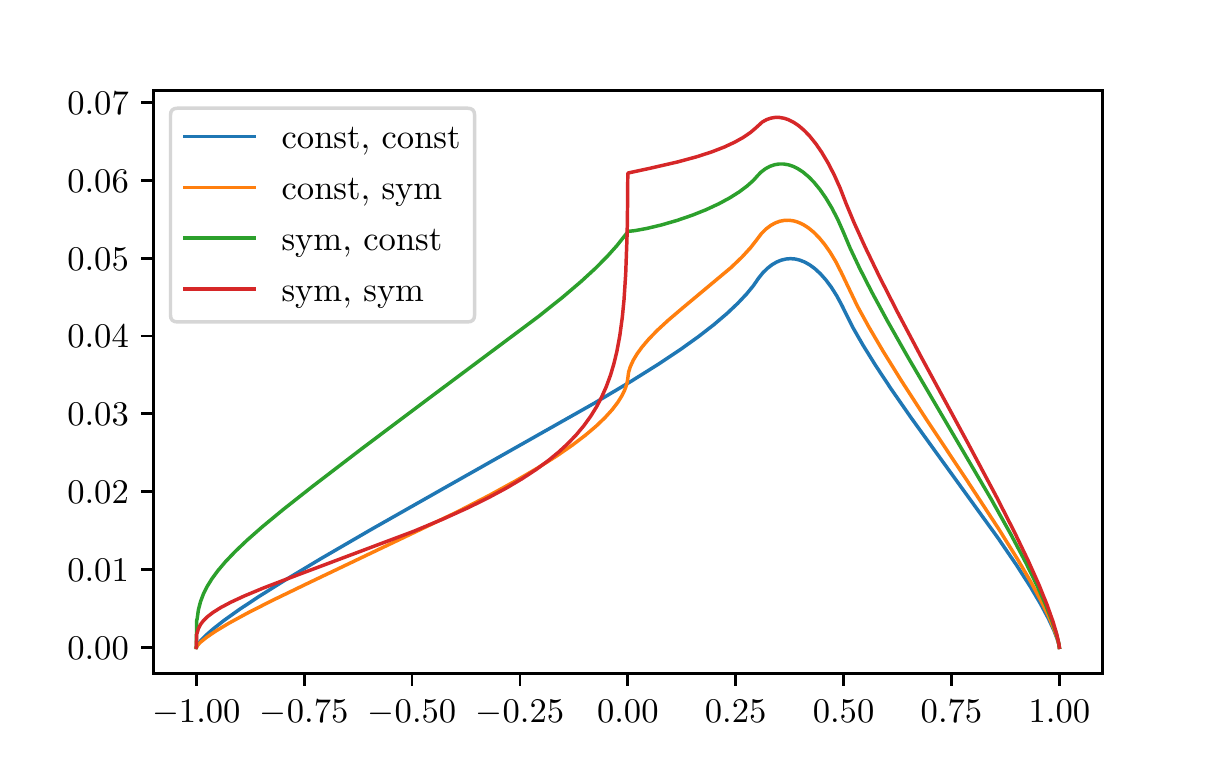}
  \caption{Solutions for symmetric kernels with \(s\) and \(\phi\) given in Table~\ref{tab:sym}.}
  \label{fig:sym-solutions}
\end{figure}

\begin{figure}
  \centering
  \includegraphics{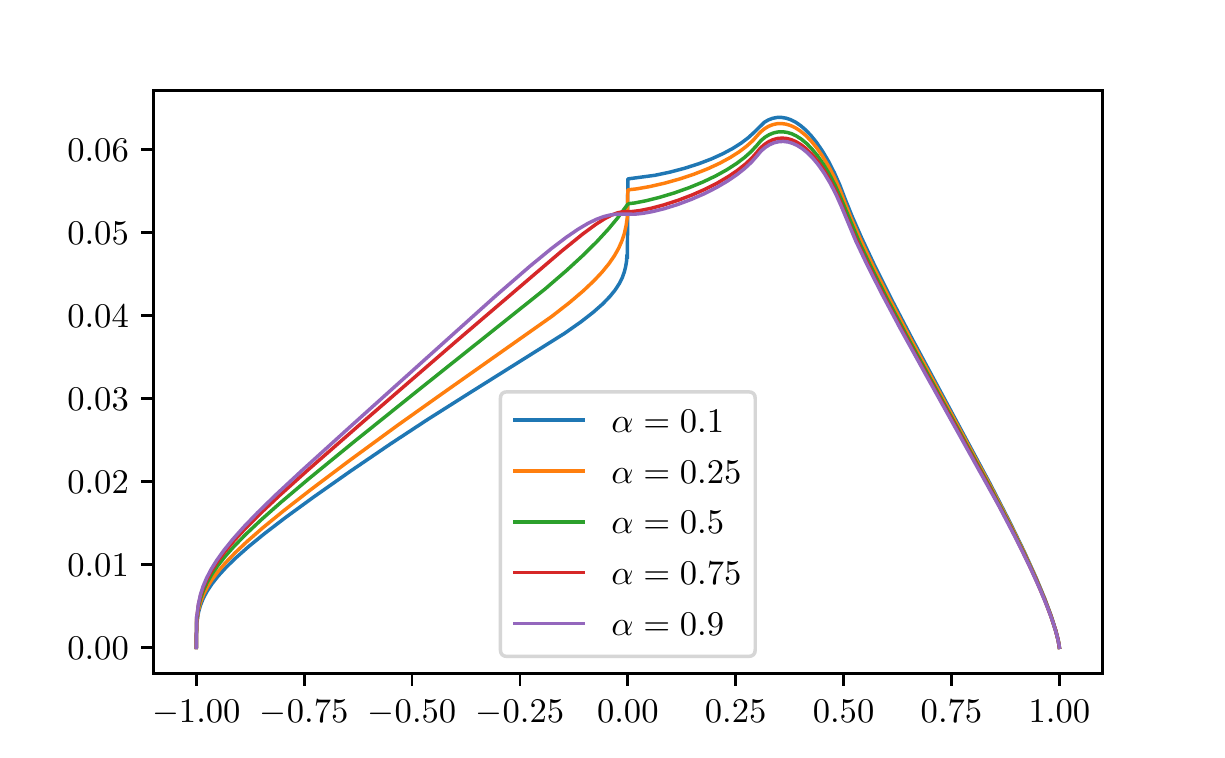}
  \caption{Solutions for symmetric kernels with \(s=\psi(\cdot;0.75,0.25,\alpha)\) for \(\alpha\in\{0.1, 0.25, 0.5, 0.75, 0.9\}\) and \(\phi=1\).}
  \label{fig:offdiag-solutions}
\end{figure}

In a second experiment, we analyze how a variation of the fractional order of interactions between the two sub-domains affects the regularity of the solution.
We choose \(s=\psi(\cdot;0.75,0.25,\alpha)\) for \(\alpha\in\{0.1, 0.25, 0.5, 0.75, 0.9\}\) and \(\phi=1\) and report the corresponding solutions in Figure~\ref{fig:offdiag-solutions}; we observe that the choice of the fractional order for the interactions impacts the entity of the jump of the derivative at \(x=0\).

\begin{figure}
  \centering
  \includegraphics{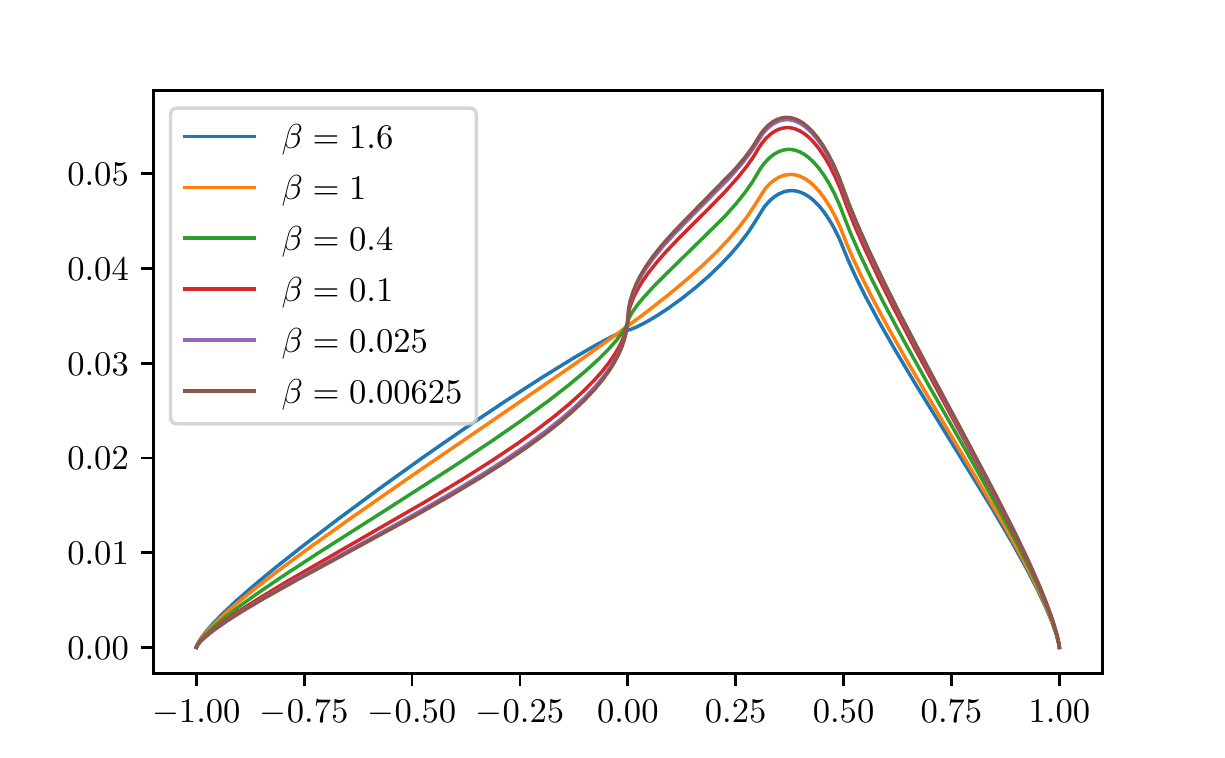}
  \caption{Solutions for symmetric kernels with \(s=0.75\) and \(\phi=\psi(\cdot;0.75,0.25,\beta)\) for \(\beta\in\{0.00625,0.025,0.1,0.4,1,1.6\}\).}
  \label{fig:offdiag-phi-solutions}
\end{figure}

In a third experiment, we analyze the impact of the parameter $\phi$.
We keep \(s=0.75\) fixed and vary the interaction between the two sub-domains by setting \(\phi=\psi(\cdot;0.75,0.25,\beta)\) for \(\beta\in\{0.00625,0.025,0.1,0.4,1,1.6\}\).
Results are reported in Figure~\ref{fig:offdiag-phi-solutions}; we observe that while \(\beta\) affects the steepness of the transition from \(\xb<0\) to \(\xb>0\), the derivative of the solution does not display any jumps.

\begin{table}
  \centering
  \begin{tabular}{lll}
    identifier & \(s\) & \(\phi\) \\
    \hline
    const, const & \(0.75\) & \(1\) \\
    const, nonsym & \(0.75\) & \(\eta(\cdot;1,0.1,0.1,200)\) \\
    nonsym, const & \(\eta(\cdot;0.75,0.25,0.1,200)\) & \(1\) \\
    nonsym, nonsym & \(\eta(\cdot;0.75,0.25,0.1,200)\) & \(\eta(\cdot;1,0.1,0.1,200)\) \\
  \end{tabular}
  \caption{Non-symmetric kernels}
  \label{tab:nonsym}
\end{table}
\begin{figure}
  \centering
  \includegraphics{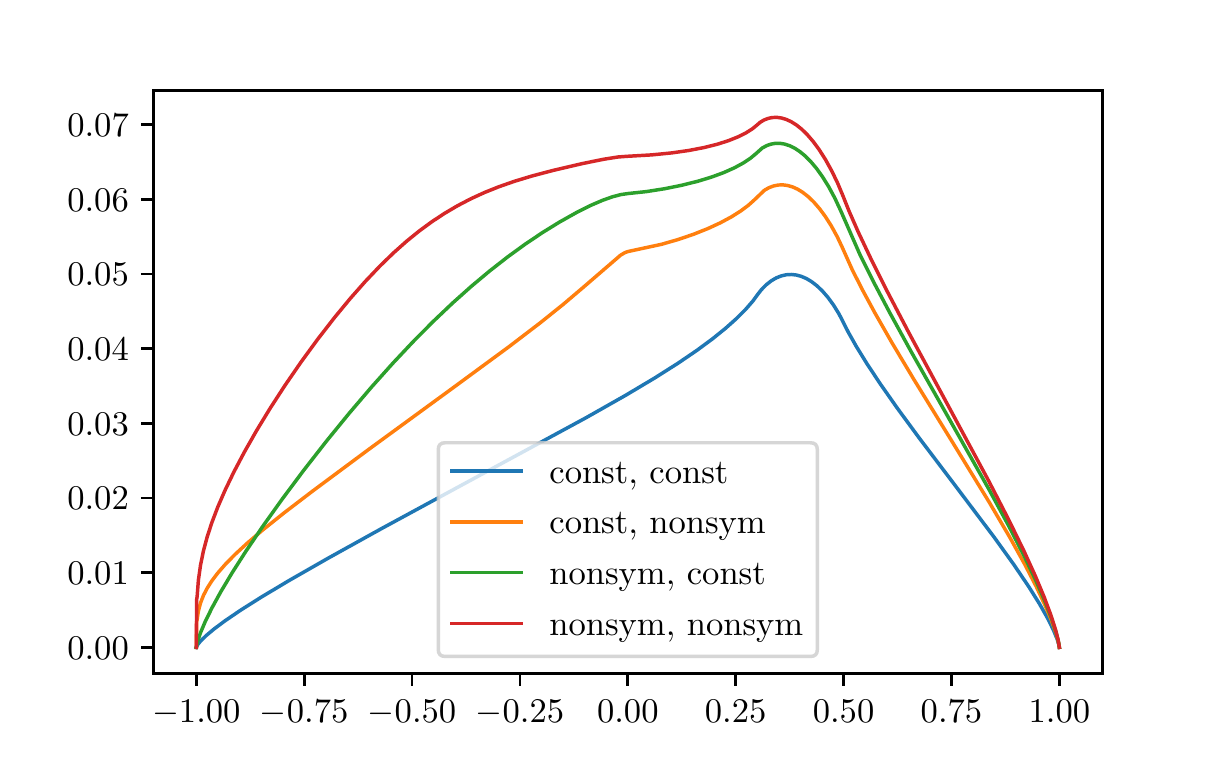}
  \caption{Solutions for non-symmetric kernels with \(s\) and \(\phi\) given in Table~\ref{tab:nonsym}.}
  \label{fig:nonsym-solutions}
\end{figure}

With the purpose of determining the impact of a nonsymmetric kernel on the qualitative behavior of the solutions, in a forth set of experiments, we consider non-symmetric kernels, as given in Table~\ref{tab:nonsym}, and report the corresponding solutions in Figure~\ref{fig:nonsym-solutions}.
These results do not highlight a significant change in the qualitative behavior of $u$ across the interface.

We conclude that the restriction to constant fractional orders \(s\) limits the types of solution behavior that can be recovered, even when allowing for a variable coefficient function \(\phi\).
On the other hand, using nonsymmetric kernels does not appear to lead to significantly different solution behavior.
Given the significant drawbacks in terms of implementation difficulties posed by nonsymmetric kernels, we therefore restrict ourselves to symmetric kernels with constant \(\phi\) in what follows.

\subsection{Convergence with respect to the mesh size}\label{sec:h-convergence}
We consider the convergence of solutions with respect to the mesh size $h$.
In light of the results presented in the previous section, we only focus on symmetric kernel functions $\gamma$.
Since, in general, no closed-form analytic solutions of \eqref{eq:strong-single-general} are known, we compute errors with respect to a reference solution \(u_{\underline{h}}\) obtained on a fine mesh.
We compute the energy and $L^2$ norms as follows:
\begin{align*}
  \opnorm{u_{h}-u_{\underline{h}}}^{2} &= \mcA(u_{h}-u_{\underline{h}},u_{h}-u_{\underline{h}}) = \int_{\Omega}f(u_{h}-u_{\underline{h}}) \,d\xb, \\
  \norm{u_{h}-u_{\underline{h}}}_{L^{2}(\Omega)}^{2} &= \int_{\Omega}(u_{h}-u_{\underline{h}})^{2} \,d\xb.
\end{align*}

Firstly, we consider the one-dimensional case where \(\Omega=(-1,1)\) and \(\delta=1\), with parameters \(s(\xb,\yb)=\psi(\xb,\yb;0.75,0.25,0.5)\), \(\phi\equiv 1\), \(f\equiv1\).
We display the obtained errors in Figure~\ref{fig:convergence1d}.
We observe that \(\opnorm{u_{h}-u_{\underline{h}}} \sim h^{1/2}\), whereas \(\norm{u_{h}-u_{\underline{h}}}_{L^{2}(\Omega)}\sim h^{3/4}\).
The apparent speedup for smaller values of the mesh size \(h\) are due to the fact that we are comparing against \(u_{\underline{h}}\), and not against the unknown exact solution \(u\).
It is known\cite{BG2019vi} that, for constant fractional kernels, the energy norm error converges as \(\mathcal{O}(h^{1/2-\varepsilon})\) for any \(\varepsilon>0\), whereas the \(L^{2}\)-error converges as \(\mathcal{O}(h^{\min\{1,1/2+s\}-\varepsilon})\).
Our numerical results suggest that this result can be generalized to \(\mathcal{O}(h^{1/2-\varepsilon})\) convergence in energy norm and \(\mathcal{O}(h^{\min\{1,1/2+\underline{s}\}-\varepsilon})\) convergence in \(L^{2}\)-norm for the case of variable fractional order.
Here, $\underline s$ is defined as in \eqref{eq:s-assumptions}.

\begin{figure}
  \centering
  \includegraphics[]{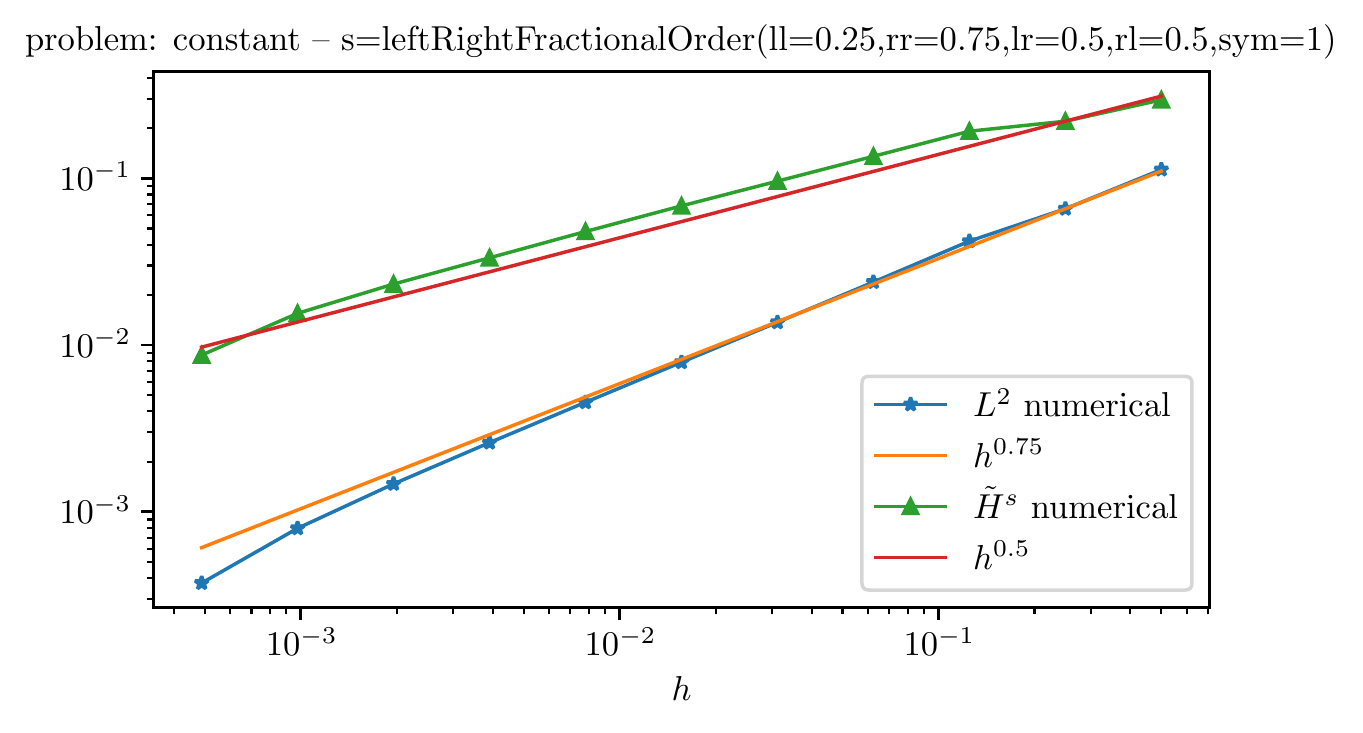}
  \caption{Convergence in \(L^{2}\) and energy norm for a one-dimensional example.}
  \label{fig:convergence1d}
\end{figure}

In a second convergence experiment, we consider the two-dimensional case of \(\Omega=(-1,1)^{2}\), and \(\delta=1/2\), with \(\phi\equiv1\), \(f\equiv1\) and a fractional orders \(s\) with four layers:
\begin{align}\label{eq:fracOrderLayers}
  s(\xb,\yb)&= \frac{1}{2}(\sigma(\xb_{1}) + \sigma(\yb_{1})) &\text{with}&&
  \sigma(z)=
  \begin{cases}
    1/5 & \text{if } z < -1/2, \\
    2/5 & \text{if } -1/2 \leq z < 0, \\
    3/5 & \text{if } 0 \leq z < 1/2, \\
    4/5 & \text{if } 1/2\leq z.
  \end{cases}
\end{align}
This definition is such that the resulting configuration mimics the realistic setting displayed in Figure \ref{fig:interface-domain-config}, right.
The reference solution for \(\underline{h}=0.01\), \(n\approx 65,000\) is displayed in Figure~\ref{fig:layers2d}.
As expected, it can be observed that the solution behavior is more diffusive for layers with larger fractional order.
Convergence results are reported in Figure~\ref{fig:convergence2d}, where we display the energy and \(L^{2}\)-norm of the discretization error with respect to a reference solution.
Similar to the one-dimensional case above, the observed rates suggest that the convergence results for constant kernels can be generalized to the variable coefficient case.

\begin{figure}
  \centering
  \includegraphics[]{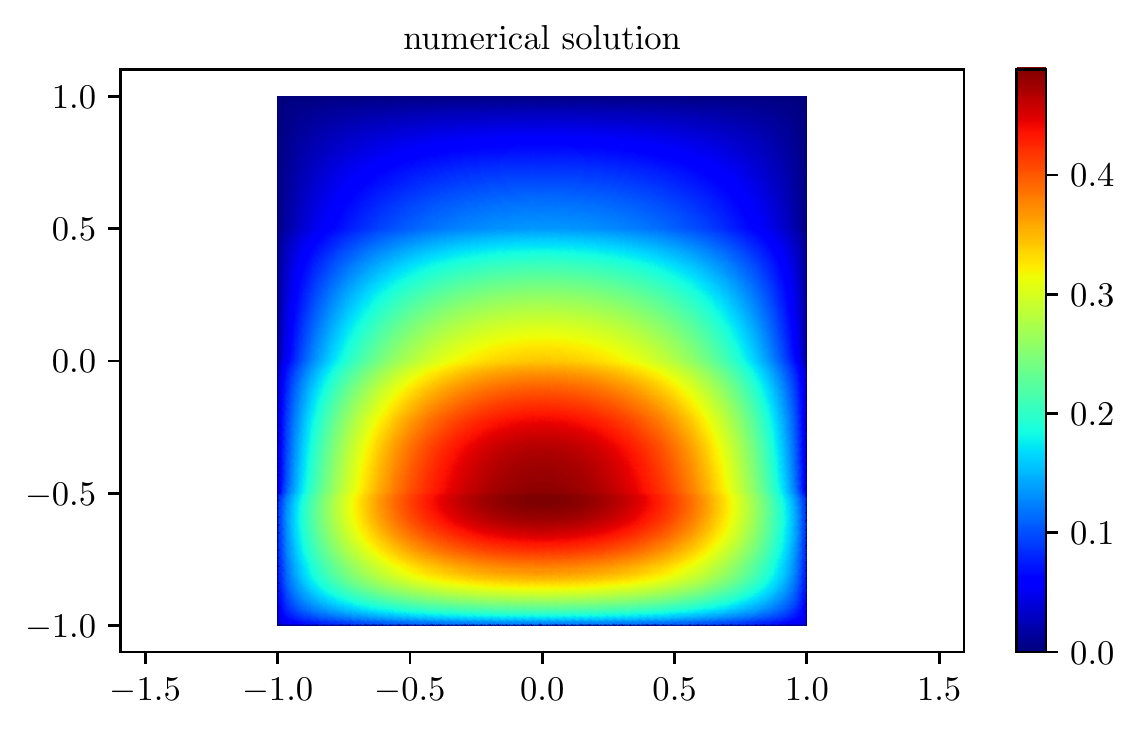}
  \caption{Solution for constant forcing, finite horizon and fractional order \eqref{eq:fracOrderLayers} with four layers.}
  \label{fig:layers2d}
\end{figure}

\begin{figure}
  \centering
  \includegraphics[]{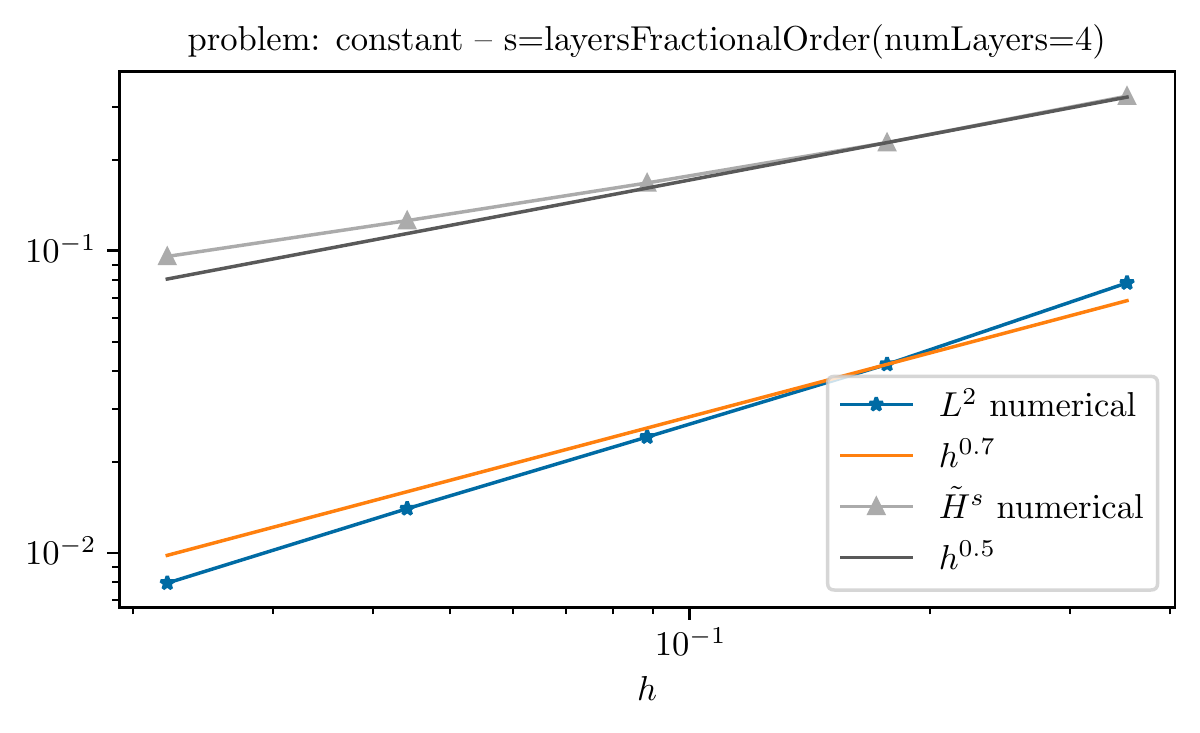}
  \caption{Convergence in \(L^{2}\) and energy norm for a two-dimensional example with four material layers.}
  \label{fig:convergence2d}
\end{figure}

\subsection{Examples of more complex configurations}\label{sec:island}

We consider a configuration that mimics the realistic setting of Figure \ref{fig:interface-domain-config}, left. Specifically, we consider the two-dimensional case of \(\Omega=B_{1}(\mathbf{0})\) and infinite horizon \(\delta=\infty\). We set \(\phi\equiv1\) and define the fractional order \(s\) by inclusions, i.e.
\begin{align}\label{eq:fracOrderIslands}
  s(\xb,\yb)&=
              \begin{cases}
                0.25 & \text{if } \abs{\xb_{i}}, \abs{\yb_{i}} \in [0.1, 0.6], i=1,2, \\
                0.75 & \text{if } \abs{\xb_{i}}, \abs{\yb_{i}} \not\in [0.1, 0.6], i=1,2, \\
                \alpha & \text{else},
              \end{cases}
\end{align}
for \(\alpha\in\{0.25,0.5,0.75\}\).
We enforce the homogeneous Dirichlet condition on \(\omgi=\mathbb{R}^{2}\setminus \Omega\) using Gauss's theorem, see~\cite{AinsworthGlusa2018}. Note that this shortcut can be applied because, for a given \(\xb\in\Omega\), \(s(\xb,\yb)\) and \(s(\yb,\xb)\) are constant for all \(\yb\in\omgi\).
Results are reported in Figure \ref{fig:island-phi1}; there, we observe that the solution behavior is less diffusive in the inclusion regions than in the surrounding domain.
Also, as in the one-dimensional test cases, we observe that the choice of the fractional order for the interactions impacts the entity of the jump of the derivative at the interface between the inclusions and the rest of the domain.
This fact clearly emphasizes the need of identification methods for variable-order fractional models.
In particular, in this case, it would be sufficient to simply learn the interaction order $\alpha$ to better predict the behavior of the solution across the interface.

\begin{figure}
  \centering
  \includegraphics[]{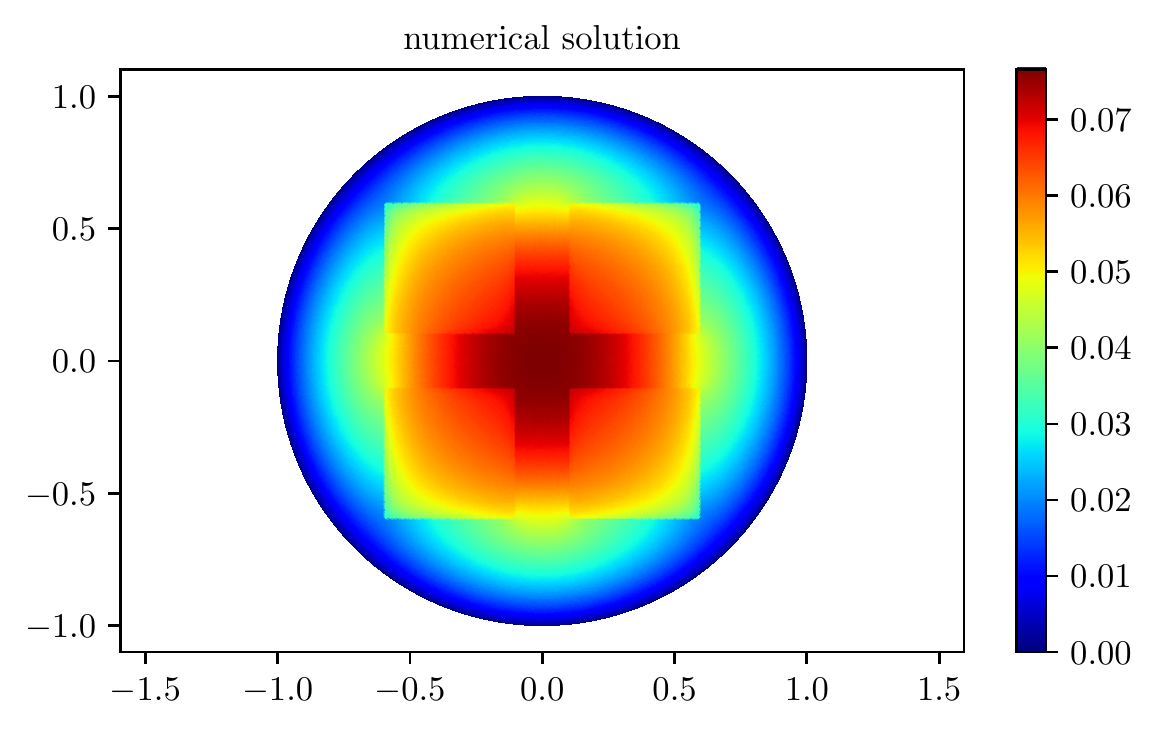}
  \includegraphics[]{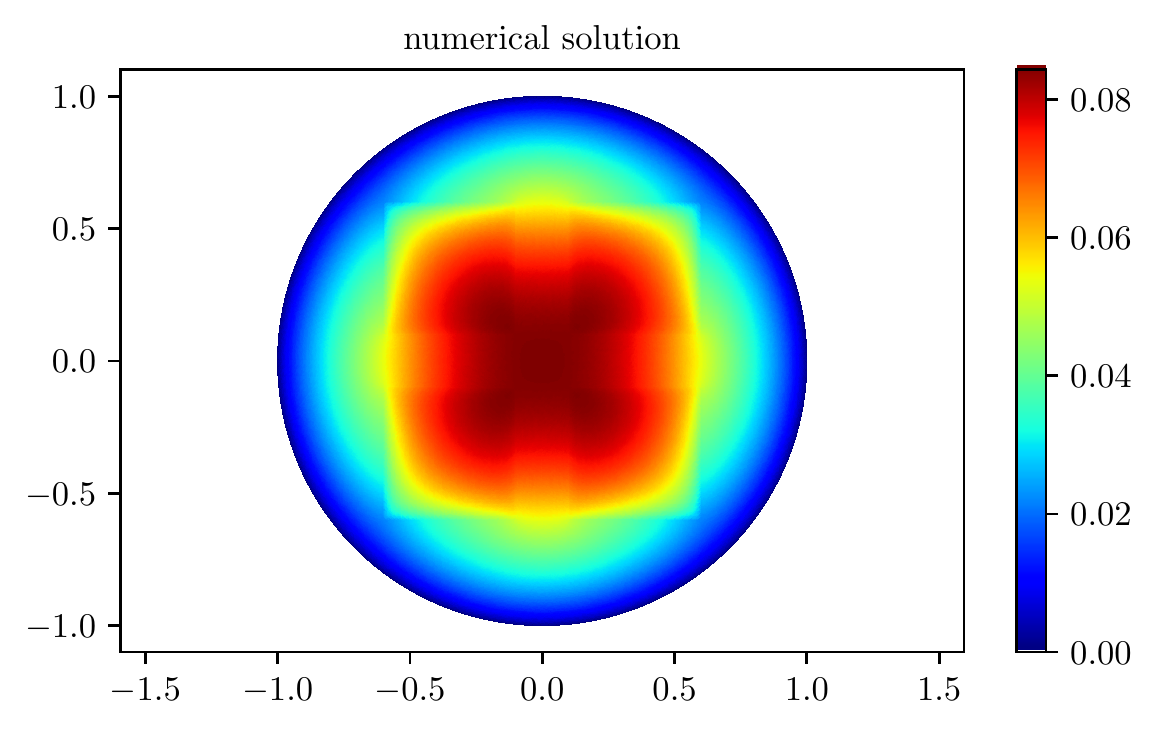}
  \includegraphics[]{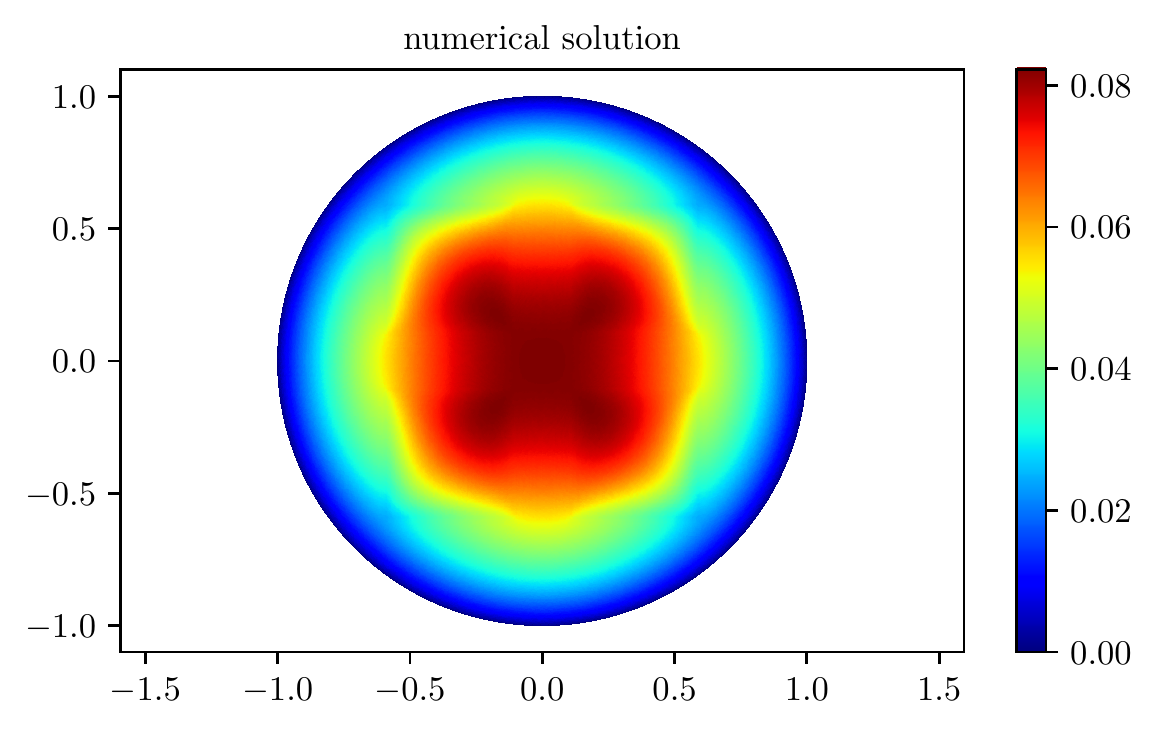}
  \caption{
    Solution for constant forcing, infinite horizon and fractional order \eqref{eq:fracOrderIslands} with inclusions.
    From top to bottom: \(\alpha\in\{0.25,0.5,0.75\}\).
  }
  \label{fig:island-phi1}
\end{figure}

\section{Concluding remarks}\label{sec:conclusion}
This work introduces a new variable-order fractional model that, compared with state-of-the-art alternatives, features improved variability.
The latter is obtained by allowing the fractional order $s$ to be a function of both $\xb$ and $\yb$, hence allowing an agile treatment of changes in the properties of the underlying physical system.
This new model finds practical use in presence of material heterogeneities and, in particular, in presence of abrupt changes in material properties, i.e. of physical interfaces.

Our theoretical results guarantee the feasibility of the proposed operator by showing that, under certain conditions on the model parameters, the associated diffusion problem is well-posed.
In fact, we point out that by proving coercivity of the time-independent problem, we automatically guarantee well-posedness of the associated parabolic problem.
On the other hand, our computational tests not only do they illustrate the improved descriptive power of the proposed model, but they also show that solutions to doubly-variable order equation converge with the same rates as the ones corresponding to the minimum value, $\underline s$, of the fractional order.
Furthermore, our implementation of the finite-element matrix assembly proves to be efficient in presence of piecewise constant definitions of $s(\xb,\yb)$.

Clearly, the problem of finding the fractional order profile that best fits a physical system remains open and raises many challenges.
However, our proposed model and the corresponding numerical tests show that the identification problem can be reduced to determining a handful of parameters, namely, the fractional orders in each subregion and the interaction orders.
Furthermore, the dependence on both $\xb$ and $\yb$ allows us to consider symmetric kernels, simplifying significantly the finite element implementation and allowing for more efficient assembly algorithms.

\section*{Acknowledgments}

\subsection*{Financial disclosure}
MD and CG are supported by Sandia National Laboratories (SNL), SNL is a multimission laboratory managed and operated by National Technology and Engineering Solutions of Sandia, LLC., a wholly owned subsidiary of Honeywell International, Inc., for the U.S. Department of Energys National Nuclear Security Administration contract number DE-NA0003525. This work was supported through the Sandia National Laboratories Laboratory-directed Research and Development (LDRD) program, project 218318 and by the U.S. Department of Energy, Office of Advanced Scientific Computing Research under the Collaboratory on Mathematics and Physics-Informed Learning Machines for Multiscale and Multiphysics Problems (PhILMs) project. This paper describes objective technical results and analysis. Any subjective views or opinions that might be expressed in the paper do not necessarily represent the views of the U.S. Department of Energy or the United States Government. SAND Number: SAND2021-0842 O.

\subsection*{Conflict of interest}
The authors declare no potential conflict of interests.

\section*{Supporting information}
None reported.

\nocite{*}
\bibliographystyle{plain}
\bibliography{references}%
\end{document}